\documentclass[12pt,a4]{amsart}
\usepackage[centertags]{amsmath}
\usepackage{amsfonts}
\usepackage{amssymb}
\usepackage{amsthm}
\usepackage{newlfont}
\newcommand{\K}{\Bbb K}
\newcommand{\C}{\Bbb C}
\newcommand{\Z}{\Bbb Z}
\newcommand{\N}{\Bbb N}
\newcommand{\Q}{\Bbb Q}

\def\NL{\hfill\break}

\newcommand{\set}[1]{\left\{#1\right\}}

\theoremstyle{thmit} 
\newtheorem{thm}{Theorem}[section]
\newtheorem{lem}[thm]{Lemma}
\newtheorem{cor}[thm]{Corollary}
\newtheorem{prop}[thm]{Proposition}

\theoremstyle{thmrm} 
\newtheorem{exa}{Example}
\newtheorem{df}{Definition}

\newtheorem*{rk}{Remark}
\newtheorem*{oldproof}{Proof}
\renewenvironment{proof}[1][{}]{\begin{oldproof}[#1]}{\qed\end{oldproof}}

\begin{document}

\title{Ore extensions of principally quasi-Baer rings}
\maketitle

\begin{center}
{\bf Mohamed Louzari and L'moufadal Ben Yakoub}
\end{center}
{\small
\begin{center}
Dept.\ of Mathematics, Abdelmalek Essaadi University
\end{center}
\begin{center}
Faculty of sciences, B.P. 2121 Tetouan, Morocco
\end{center}
\begin{center}
mlouzari@yahoo.com, benyakoub@hotmail.com 
\end{center}}

\begin{abstract}Let $R$ be a ring and $(\sigma,\delta)$ a quasi-derivation of $R$. In this paper,
we show that if $R$ is an $(\sigma,\delta)$-skew Armendariz ring and satisfies the condition $(\mathcal{C_{\sigma}})$, then $R$ is right p.q.-Baer if and only if the Ore extension $R[x;\sigma,\delta]$ is right p.q.-Baer. As a consequence we obtain a generalization of \cite{hong/2000}.
\end{abstract}

\footnote[0]{

{2000 Mathematics Subject Classification.} 16S36.
\NL{Key words and phrases.} p.q.-Baer rings, Ore extensions,
$(\sigma,\delta)$-skew Armendariz rings, $(\sigma,\delta)$-compatible rings, $\sigma$-rigid rings.}

\section{Introduction}
Throughout this paper, $R$ denotes an associative ring with unity. For a subset $X$ of $R$, $r_R(X)=\{a\in R|Xa=0\}$ and $\ell_R(X)=\{a\in R|aX=0\}$ will stand for the right and the left annihilator of $X$ in $R$ respectively.
By Kaplansky \cite{Kaplansky}, a right annihilator of $X$ is always a right ideal, and if $X$ is a right ideal then $r_R(X)$ is a two-sided ideal. An Ore extension of a ring $R$ is denoted by $R[x;\sigma,\delta]$, where $\sigma$ is an endomorphism of $R$ and $\delta$ is a $\sigma$-derivation, i.e., $\delta\colon R\rightarrow R$ is an additive map such that $\delta(ab)=\sigma(a)\delta(b)+\delta(a)b$ for all $a,b\in R$ (the pair $(\sigma,\delta)$ is also called a quasi-derivation of $R$). Recall that elements of $R[x;\sigma,\delta]$ are polynomials in $x$ with coefficients written on the left. Multiplication in $R[x;\sigma,\delta]$ is given by the multiplication in $R$ and the condition $xa=\sigma(a)x+\delta(a)$, for all $a\in R$. We say that a subset $X$ of $R$ is $(\sigma,\delta)$-{\it stable} if
$\sigma(X)\subseteq X$ and $\delta(X)\subseteq X$. Recall that a ring $R$ is ({\it quasi})-{\it Baer} if the right
annihilator of every (right ideal) nonempty subset of $R$ is generated by an idempotent. Kaplansky \cite{Kaplansky}, introduced Baer rings to abstract various property of $AW^*$-algebras and Von Neumann algebras. Clark \cite{clark}, defined quasi-Baer rings and used them to characterize when a finite dimensional algebra with unity over an algebraically closed field is isomorphic to a twisted matrix units semigroup algebra. Another generalization of Baer rings are the p.p.-rings. A ring $R$ is a right (respectively, left) {\it p.p.-ring} if the right (respectively, left) annihilator of an element of $R$ is generated by an idempotent (right p.p.-rings are also known as the right Rickart rings). $R$ is called a {\it p.p.-ring} if it is both right and left p.p.-ring. Birkenmeier et al. \cite{birk/pqBaer}, introduced
principally quasi-Baer rings and used them to generalize many results on reduced p.p.-rings. A ring is called {\it right principally quasi-Baer} (or simply {\it right p.q.-Baer}) if the right annihilator of a principal right ideal is generated by an idempotent. Similarly, left p.q.-Baer rings can be defined. A ring $R$ is called {\it p.q.-Baer} if it is both right and left p.q.-Baer. For more details and examples of right p.q.-Baer rings, see Birkenmeier et al. \cite{birk/pqBaer}.\par From Birkenmeier et al. \cite{birk/polyExt}, an idempotent $e\in R$ is left (respectively, right) {\it
semicentral} in $R$ if $ere=re$ (respectively, $ere=er$), for all $r\in R$. Equivalently, $e^2=e\in R$ is left (respectively, right) semicentral if $eR$ (respectively, $Re$) is an ideal of $R$. Since the right annihilator of a right ideal is an ideal, we see that the right annihilator of a right ideal is generated by a left semicentral in a quasi-Baer (p.q.-Baer) ring. We use $\mathcal{S}_\ell(R)$ and $\mathcal{S}_r(R)$ for the sets of all left and right semicentral idempotents, respectively. Also note $\mathcal{S}_\ell(R)\cap \mathcal{S}_r(R)=\mathcal{B}(R)$, where $\mathcal{B}(R)$ is the set
of all central idempotents of $R$. If $R$ is a semiprime ring then $\mathcal{S}_\ell(R)=\mathcal{S}_r(R)=\mathcal{B}(R)$. Recall that $R$ is a {\it reduced} ring if it has no nonzero nilpotent elements. A ring $R$ is {\it abelian} if every idempotent of $R$ is central. We can easily observe that every reduced ring is abelian. \par According to Krempa \cite{krempa}, an endomorphism $\sigma$ of a ring $R$ is
called {\it rigid} if $a\sigma(a)=0$ implies $a=0$ for all $a\in R$. We call a ring $R$ $\sigma$-{\it rigid} if there exists a rigid endomorphism $\sigma$ of $R$. Note that any rigid endomorphism of a ring $R$ is a monomorphism and $\sigma$-rigid rings are reduced by Hong et al. \cite{hong/2000}. A ring $R$ is called {\it Armendariz} (respectively, $\sigma$-{\it skew Armendariz}) if whenever polynomials $f=\sum_{i=0}^{n}a_ix^i,\;g=\sum_{j=0}^{m}b_jx^j$ in $R[x]$ (respectively, $R[x;\sigma]$) satisfy $fg=0$ then $a_ib_j=0$ (respectively, $a_i\sigma^i(b_j)=0$) for each $i,j$. From Hashemi and Moussavi \cite{hashemi/skew}, a ring $R$ is called an $(\sigma,\delta)$-{\it skew Armendariz } ring if for $p=\sum_{i=0}^{n}a_ix^i$ and $q=\sum_{j=0}^{m}b_jx^j$ in $R[x;\sigma,\delta]$, $pq=0$ implies $a_ix^ib_jx^j=0$ for each $i,j$. Note that $(\sigma,\delta)$-skew Armendariz rings are generalization of $\sigma$-skew Armendariz rings, $\sigma$-rigid rings and Armendariz rings, see Hong et al. \cite{hong/2003}, for more details. Following Hashemi and Moussavi \cite{hashemi/quasi}, a ring $R$ is $\sigma$-{\it compatible} if for each $a,b\in R$, $a\sigma(b)=0\Leftrightarrow ab=0$. Moreover, $R$ is said to be $\delta$-{\it compatible} if for each $a,b\in R$, $ab=0\Rightarrow a\delta(b)=0$. If $R$ is both $\sigma$-compatible and $\delta$-compatible, we say that $R$ is $(\sigma,\delta)$-{\it compatible}.\par Birkenmeier et al. \cite[Theorem 3.1]{birk/OnpolyExt}, have proved that $R$ is right p.q.-Baer if and only if $R[x]$ is right p.q.-Baer. Hong et al. \cite[Corollary 15]{hong/2000}, have showed
that, if $R$ is $\sigma$-rigid, then $R$ is right p.q.-Baer if and only if $R[x;\sigma,\delta]$ is right p.q.-Baer. Also, Hashemi and Moussavi in \cite[Corollary 2.8]{hashemi/quasi}, have proved that under the $(\sigma,\delta)$-compatibility assumption on the ring $R$, $R$ is right p.q.-Baer if and only if $R[x;\sigma,\delta]$ is right p.q.Baer. \par In this paper, we prove that if $R$ is an $(\sigma,\delta)$-skew Armendariz ring and satisfies the condition $(\mathcal{C_{\sigma}})$ (see Definition \ref{def}), then $R$ is right p.q.-Baer if and only if the Ore extension $R[x;\sigma,\delta]$ is right p.q.-Baer. If  $R$ is a $\sigma$-rigid ring then $R$ is $(\sigma,\delta)$-skew Armendariz ring and satisfies the condition $(\mathcal{C_{\sigma}})$. So that we obtain a generalization of \cite[Corollary 15]{hong/2000}.

\section{Preliminaries and Examples}

For any $0\leq i\leq j\;(i,j\in \N)$, $f_i^j\in End(R,+)$ will denote the map which is the sum of all possible words in $\sigma,\delta$ built with $i$ letters $\sigma$ and $j-i$ letters $\delta$ (e.g., $f_n^n=\sigma^n$ and $f_0^n=\delta^n, n\in \N $). For any $n\in \N$ and $r\in R$ we have $x^nr=\sum_{i=0}^nf_i^n(r)x^i$ in the ring $R[x;\sigma,\delta]$ (see \cite[Lemma 4.1]{lam}).

\begin{lem}\label{lem2}Let $R$ be a ring and $a,b,c\in R$ such that $b\in
r_R(cR)=eR$ and $Re$ is $(\sigma,\delta)$-stable for some $e\in
\mathcal{S}_{\ell}(R)$. Then:\NL
$(i)$ $c\sigma(ab)=c\delta(ab)=0$;\NL $(ii)$ $cf_k^j(ab)=0$, for all
$0\leq k\leq j\;(k,j\in \N)$.
\end{lem}

\begin{proof}$(i)$ $c\sigma(ab)=c\sigma(a)\sigma(b)$, but $b=eb$, so
$c\sigma(ab)=c\sigma(a)\sigma(e)\sigma(b)$, since
$\sigma(e)=\sigma(e)e$. Then
$c\sigma(ab)=c\sigma(a)\sigma(e)e\sigma(b)=0$, because $e\in
r_R(cR)$. Also
$c\delta(ab)=c\delta(aeb)=c\sigma(ae)\delta(b)+c\delta(ae)b$, but
$c\delta(ae)b=0$. So
$c\delta(ab)=c\sigma(ae)\delta(b)=c\sigma(a)\sigma(e)e\delta(b)=0$.\NL$(ii)$
It follows from $(i)$.
\end{proof}

\begin{df}\label{def}Let $\sigma$ be an endomorphism of a ring $R$. We say that $R$
satisfies the condition $(\mathcal{C_{\sigma}})$ if whenever $a\sigma(b)=0$ with $a,b\in R$, then $ab=0$.
\end{df}

\begin{lem}\label{lem5}Let $\sigma$ be an endomorphism of a ring $R$.
The following are equivalent:
\NL$(i)$ $R$ satisfies $(\mathcal{C_{\sigma}})$ and reduced;
\NL$(ii)$ $R$ is $\sigma$-rigid.
\end{lem}

\begin{proof}Let $a\in R$ such that $a\sigma(a)=0$ then $a^2=0$, since $R$ is reduced so
$a=0$. Conversely, let $a,b\in R$ such that $a\sigma(b)=0$ then
$ba\sigma(ba)=0$, since $R$ is $\sigma$-rigid (so reduced) then
$ba=ab=0$.
\end{proof}

\begin{lem}\label{lemma1}Let $R$ be a ring, $\sigma$ an endomorphism of $R$
and $\delta$ be a $\sigma$-derivation of $R$. If $R$ is
$(\sigma,\delta)$-compatible. Then for $a,b\in R$, $ab=0$ implies
$af_i^j(b)=0$ for all $j\geq i\geq 0$.
\end{lem}
\begin{proof}If $ab=0$, then $a\sigma^i(b)=a\delta^j(b)=0$ for all
$i\geq0$ and $j\geq 0$, because $R$ is $(\sigma,\delta)$-compatible.
Then $af_i^j(b)=0$ for all $i,j$.
\end{proof}

There is an example of a ring $R$ and an endomorphism $\sigma$ of
$R$ such that $R$ is $\sigma$-skew Armendariz and $R$ is not
$\sigma$-compatible.

\begin{exa}\label{ex2.1}Consider a ring of polynomials over $\Z_2$,
$R=\Z_2[x]$. Let $\sigma\colon R\rightarrow R$ be an endomorphism
defined by $\sigma(f(x))=f(0)$. Then: \NL $(i)$ $R$ is not
$\sigma$-compatible. Let $f=\overline{1}+x$, $g=x\in R$, we have
$fg=(\overline{1}+x)x\neq 0$, however
$f\sigma(g)=(\overline{1}+x)\sigma(x)=0$.\NL $(ii)$ $R$ is
$\sigma$-skew Armendariz \cite[Example~5]{hong/2003}.
\end{exa}

In the next example, $S=R/I$ is a ring and $\overline{\sigma}$ an
endomorphism of $S$ such that $S$ is $\overline{\sigma}$-compatible
and not $\overline{\sigma}$-skew Armendariz.

\begin{exa}\label{ex4} Let $\Z$ be the ring of integers and $\Z_4$ be
the ring of integers modulo 4. Consider the ring
$$R=\set{\begin{pmatrix}
  a & \overline{b} \\
  0 & a
\end{pmatrix}| a\in\Z\;,\overline{b}\in \Z_4}.$$  Let $\sigma\colon R\rightarrow R$ be an
endomorphism defined by $\sigma\left(\begin{pmatrix}
  a & \overline{b} \\
  0 & a
\end{pmatrix}\right)=\begin{pmatrix}
  a & -\overline{b} \\
  0 & a
\end{pmatrix}$.\NL Take the ideal $I=\set{\begin{pmatrix}
  a & \overline{0} \\
  0 &a
\end{pmatrix}| a\in 4\Z}$ of $R$. Consider the factor ring $$R/I\cong\set{\begin{pmatrix}
  \overline{a} & \overline{b} \\
  0 &\overline{a}
\end{pmatrix}| \overline{a},\overline{b}\in 4\Z}.$$ \NL$(i)$ $R/I$ is not
$\overline{\sigma}$-skew Armendariz. In fact, $\left(\begin{pmatrix}
 \overline{2} & \overline{0} \\
  0 & \overline{2}
\end{pmatrix}+\begin{pmatrix}
 \overline{2} & \overline{1} \\
  0 & \overline{2}
\end{pmatrix}x\right)^2=0\in(R/I)[x;\overline{\sigma}]$, but $\begin{pmatrix}
 \overline{2} & \overline{1} \\
  0 & \overline{2}
\end{pmatrix}\overline{\sigma}\begin{pmatrix}
 \overline{2} & \overline{0} \\
  0 & \overline{2}
\end{pmatrix}\neq 0$.\NL$(ii)$ $R/I$ is
$\overline{\sigma}$-compatible. Let $A=\begin{pmatrix}
 \overline{a} & \overline{b} \\
  0 & \overline{a}
\end{pmatrix}\;,B=\begin{pmatrix}
 \overline{a'} & \overline{b'} \\
  0 & \overline{a'}
\end{pmatrix}\in R/I$. If $AB=0$ then $\overline{aa'}=0$ and
$\overline{ab'}=\overline{ba'}=0$, so that
$A\overline{\sigma}(B)=0$. The same for the converse. Therefore
$R/I$ is $\overline{\sigma}$-compatible.
\end{exa}

\begin{exa}\label{ex2.2}Consider the ring $$R=\set{\begin{pmatrix}
  a & t \\
  0 & a
\end{pmatrix}| a\in \Z\;,t\in \Q},$$ where $\Z$ and $\Q$ are the
set of all integers and all rational numbers, respectively. The ring
$R$ is commutative, let $\sigma\colon R\rightarrow R$ be an
automorphism defined by $\sigma\left(\begin{pmatrix}
  a & t \\
  0 & a
\end{pmatrix}\right)=\begin{pmatrix}
  a & t/2 \\
  0 & a
\end{pmatrix}$.\NL
$(i)$ $R$ is not $\sigma$-rigid. $\begin{pmatrix}
  0 & t \\
  0 & 0
\end{pmatrix}\sigma\left(\begin{pmatrix}
  0 & t \\
  0 & 0
\end{pmatrix}\right)=0$, but $\begin{pmatrix}
  0 & t \\
  0 & 0
\end{pmatrix}\neq 0$, if $t\neq 0$.\NL $(ii)$ $\sigma(Re)\subseteq
Re$ for all $e\in \mathcal{S}_{\ell}(R)$. $R$ has only two
idempotents, $e_0=\begin{pmatrix}
  0 & 0 \\
  0 & 0
\end{pmatrix}$ end $e_1=\begin{pmatrix}
 1 & 0 \\
  0 & 1
\end{pmatrix}$, let $r=\begin{pmatrix}
  a & t \\
  0 & a
\end{pmatrix}\in R$, we have $\sigma(re_0)\in Re_0$ and $\sigma(re_1)\in
Re_1$. \NL $(iii)$ $R$ satisfies the condition $(\mathcal{C_\sigma})$. Let $\left(%
\begin{array}{cc}
  a & t \\
  0 & a \\
\end{array}%
\right)$ and $\left(%
\begin{array}{cc}
  b & x \\
  0 & b \\
\end{array}%
\right)\in R$ such that $$\left(%
\begin{array}{cc}
  a & t \\
  0 & a \\
\end{array}%
\right)\sigma\left(\left(%
\begin{array}{cc}
  b & x \\
  0 & b \\
\end{array}%
\right)\right)=0,$$ hence $ab=0=ax/2+tb$, so $a=0$ or $b=0$. In each
case,
$ax+tb=0$, hence $\left(%
\begin{array}{cc}
  a & t \\
  0 & a \\
\end{array}%
\right)\left(%
\begin{array}{cc}
  b & x \\
  0 & b \\
\end{array}%
\right)=0$. Therefore $R$ satisfies $(\mathcal{C_\sigma})$.
\NL $(iv)$ $R$ is $\sigma$-skew Armendariz \cite[Example 1]{hong/2000}.
\end{exa}

\section{Ore extensions over right p.q.-Baer rings}

The principally quasi-Baerness of a ring $R$ do not inherit the Ore extensions of $R$. The following example
shows that, there exists an endomorphism $\sigma$ of a ring $R$ such that $R$ is right p.q.-Baer, $Re$ is $\sigma$-stable for all $e\in \mathcal{S_{\ell}}(R)$ and not satisfying $(\mathcal{C_{\sigma}})$, but
$R[x;\sigma]$ is not right p.q.-Baer.

\begin{exa}\label{ex3.1}Let $\K$ be a field and $R=\K[t]$ a polynomial ring over $\K$
with the endomorphism $\sigma$ given by $\sigma(f(t))=f(0)$  for all
$f(t)\in R$. Then $R$ is a principal ideal domain so right
p.q.-Baer.\NL $(i)$ $R[x;\sigma]$ is not right p.q.-Baer. Consider a
right ideal $xR[x;\sigma]$. Then
$x\{f_0(t)+f_1(t)x+\cdots+f_n(t)x^n\}=f_0(0)x+f_1(0)x^2+\cdots+f_n(0)x^{n+1}$
for all $f_0(t)+f_1(t)x+\cdots+f_n(t)x^n\in R[x;\sigma]$ and hence
$xR[x;\sigma]=\{a_1x+a_2x^2+\cdots+a_nx^n|\;n\in\N,\;
a_i\in\K\;(i=0,1,\cdots,n)\}$. Note that $R[x;\sigma]$ has only two
idempotents 0 and 1 by simple computation. Since
$(a_1x+a_2x^2+\cdots+a_nx^n)1=a_1x+a_2x^2+\cdots+a_nx^n\neq 0$ for
some nonzero element $a_1x+a_2x^2+\cdots+a_nx^n\in xR[x;\sigma]$, we
get $1\notin r_{R[x;\sigma]}(xR[x;\sigma])$ and so
$r_{R[x;\sigma]}(xR[x;\sigma])\neq R[x;\sigma]$. Also, since
$(a_1x+a_2x^2+\cdots+a_nx^n)t=0$ for all
$a_1x+a_2x^2+\cdots+a_nx^n\in xR[x;\sigma],\; t\in
r_{R[x;\sigma]}(xR[x;\sigma])$ and hence
$r_{R[x;\sigma]}(xR[x;\sigma])\neq 0$. Thus
$r_{R[x;\sigma]}(xR[x;\sigma])$ is not generated by an idempotent.
Therefore $R[x;\sigma]$ is not a right p.q.-Baer ring, \cite[Example
2.8]{Birk/someskew}. \NL$(ii)$ $R$ does not satisfy the condition
$(\mathcal{C_{\sigma}})$. Take $f=a_0+a_1t+a_2t^2+\cdots+a_nt^n$ and
$g=b_1t+b_2t^2+\cdots+b_mt^m$, since $g(0)=0$ so, $f\sigma(g)=0$,
but $fg\neq 0$.\NL$(iii)$ $R$ has only two idempotents $0$ and $1$
so $Re$ is $\sigma$-stable for all $e\in\mathcal{S}_\ell(R)$.
\end{exa}

\begin{prop}\label{proprincipalart2}Let $R$ be a ring and $(\sigma,\delta)$ a quasi-derivation of $R$.
Assume that $Re$ is $(\sigma,\delta)$-stable for all $e\in\mathcal{S}_\ell(R)$ and $R$
satisfies the condition $(\mathcal{C_\sigma})$. If $R$ is right p.q.-Baer then so is
$R[x;\sigma,\delta]$.
\end{prop}

\begin{proof}The idea of proof is similar to that  of \cite[Theorem 3.1]{birk/OnpolyExt}.
Let $S=R[x;\sigma,\delta]$ and
$p(x)=c_0+c_1x+\cdots+c_nx^n\in S$. There is
$e_i\in\mathcal{S}_\ell(R)$ such that $r_R(c_iR)=e_iR$, for
$i=0,1,\cdots,n$. Let $e=e_ne_{n-1}\cdots e_0$, then
$e\in\mathcal{S}_\ell(R)$ and
$eR=\bigcap_{i=0}^nr_R(c_iR)$.
\NL{\bf Claim 1}. $eS\subseteq r_S(p(x)S)$.
\NL Let $\varphi(x)=a_0+a_1x+\cdots+a_mx^m\in
S$, we have
$$p(x)\varphi(x)e=(\sum_{i=0}^nc_ix^i)(\sum_{k=0}^m(\sum_{j=k}^
m a_kf_k^j(e)x^k)),$$ since $Re$ is $f_k^j$-stable $(0\leq k\leq j)$,
we have $f_k^j(e)=u_k^je$ for some $u_k^j\in R$ $(0\leq k\leq j)$.
So
$p(x)\varphi(x)e=(\sum_{i=0}^nc_ix^i)(\sum_{k=0}^m(\sum_{j=k}^
m a_ku_k^je)x^k)$, if we set $\sum_{j=k}^ma_ku_k^j=\alpha_k$,
then $$
 p(x)\varphi(x)e= (\sum_{i=0}^nc_ix^i)(\sum_{k=0}^ m
\alpha_kex^k)=\sum_{i=0}^n\sum_{k=0}^m(c_i\sum_{j=0}^i
f_j^i(\alpha_ke))x^{j+k},$$ but $eR\subseteq r_R(c_iR)$, for
$i=0,1,\cdots\,n$. So $p(x)\varphi(x)e=0$. Therefore $eS\subseteq
r_S(p(x)S)$.
\NL{\bf Claim 2}. $r_S(p(x)R)\subseteq eS$.\NL Let
$\varphi(x)=a_0+a_1x+\cdots+a_mx^m\in r_S(p(x)R)$. Since
$p(x)R\varphi(x)=0$, we have $p(x)b\varphi(x)=0$ for all $b\in R$.
Thus $$\sum_{i=0}^n\sum_{k=0}^m(c_i\sum_{j=0}^if_j^i(ba_k))x^{j+k}=0.$$
So, we have the following system of equations:
$$c_n\sigma^n(ba_m)=0;\eqno(0)$$

$$c_n\sigma^n(ba_{m-1})+c_{n-1}\sigma^
{n-1}(ba_m)+c_nf_{n-1}^n(ba_m)=0;\eqno(1)$$

$$c_n\sigma^n(ba_{m-2})+c_{n-1}\sigma^{n-1}(ba_{m-1})+
c_nf_{n-1}^n(ba_{m-1})+c_{n-2}\sigma^{n-2}(ba_m)\eqno(2)$$ $$+
c_{n-1}f_{n-2}^{n-1} (ba_m)+c_nf_{n-2}^n(ba_m)=0;$$

$$c_n\sigma^n(ba_{m-3})+c_{n-1}\sigma^{n-1}(ba_{m-2})+
c_nf_{n-1}^n(ba_{m-2})+c_{n-2}\sigma^{n-2}(ba_{m-1})\eqno(3)$$
$$+c_{n-1}f_{n-2}^{n-1}(ba_{m-1})+c_nf_{n-2}^n(ba_{m-1})+c_{n-3}
\sigma^{n-3}(ba_m)+c_{n-2}f_{n-3}^{n-2}(ba_m)$$
$$+c_{n-1}f_{n-3}^{n-1}(ba_m)+c_nf_{n-3}^n(ba_m)=0;$$
$$\cdots$$
$$\sum_{j+k=\ell}\;\;\sum_{i=0}^n\;
\sum_{k=0}^m(c_i\sum_{j=0}^if_j^i(ba_k))=0;\eqno(\ell)$$
$$\cdots$$ $$\sum_{i=0}^nc_i\delta^i(ba_0)=0.\eqno(n+m)$$
From  eq. $(0)$, we have $c_nba_m=0$ then $a_m\in r_R(c_nR)=e_nR$.
Since $c_nba_m=0$, so in eq. $(1)$, by Lemma \ref{lem2}, we have
$c_nf_{n-1}^n(ba_m)=0$ and eq. $(1)$ simplifies to
$$c_n\sigma^n(ba_{m-1})+c_{n-1}\sigma^
{n-1}(ba_m)=0.\eqno(1')$$ Let $s\in R$ and take $b=se_n$ in eq.
$(1')$. Then $$c_n\sigma^n(se_na_{m-1})+c_{n-1}\sigma^
{n-1}(se_na_m)=0.$$ But $c_nse_na_{m-1}=0$, so
$c_n\sigma^n(se_na_{m-1})=0$. Thus $c_{n-1}\sigma^
{n-1}(se_na_m)=0$, so $c_{n-1}se_na_m=0$ but $e_na_m=a_m$, the eq.
$(1')$ yields $c_{n-1}sa_m=0$. Hence $a_m\in r_R(c_{n-1}R)$, thus
$a_m\in e_ne_{n-1}R$ and so $c_n\sigma^n(ba_{m-1})=0$, so
$c_nba_{m-1}=0$, thus $a_{m-1}\in e_nR=r_R(c_nR)$.\par Now in eq.
$(2)$, since $c_nba_{m-1}=c_{n-1}ba_m=c_nb_m=0$, because $a_m\in
r_R(c_nR)\cap r_R(c_{n-1}R)$ and $a_{m-1}\in r_R(c_nR)$. By Lemma
\ref{lem2}, we have
$$c_nf_{n-1}^n(ba_{m-1})=c_{n-1}f_{n-2}^{n-1}
(ba_m)=c_nf_{n-2}^n(ba_m)=0,$$ because $a_m\in e_ne_{n-1}R$ and
$a_{m-1}\in e_nR$. So, eq. $(2)$, simplifies to
$$c_n\sigma^n(ba_{m-2})+c_{n-1}\sigma^{n-1}(ba_{m-1})+c_{n-2}
\sigma^{n-2}(ba_m)=0.\eqno(2')$$ In eq. $(2')$, take
$b=se_ne_{n-1}$. Then
$$c_n\sigma^n(se_ne_{n-1}a_{m-2})+c_{n-1}\sigma^{n-1}(se_ne_{n-1}a_{m-1})+c_{n-2}
\sigma^{n-2}(se_ne_{n-1}a_m)=0.$$ But
$$c_n\sigma^n(se_ne_{n-1}a_{m-2})=c_{n-1}\sigma^{n-1}(se_ne_{n-1}a_{m-1})=0.$$
Hence $c_{n-2} \sigma^{n-2}(se_ne_{n-1}a_m)=0$, so
$c_{n-2}se_ne_{n-1}a_m=c_{n-2}sa_m=0$, thus $a_m\in
r_R(c_{n-2}R)=e_{n-2}R$ and so $a_m\in e_ne_{n-1}e_{n-2}R$. The eq.
$(2')$ becomes
$$c_n\sigma^n(ba_{m-2})+c_{n-1}\sigma^{n-1}(ba_{m-1})=0.\eqno(2'')$$
Take $b=se_n$ in eq. $(2'')$, so
$c_n\sigma^n(se_na_{m-2})+c_{n-1}\sigma^{n-1}(se_na_{m-1})=0,$ then
$c_{n-1}\sigma^{n-1}(se_na_{m-1})=0$ because
$c_n\sigma^n(se_na_{m-2})=0$, thus
$$c_{n-1}se_na_{m-1}=c_{n-1}sa_{m-1}=0,$$ then $a_{m-1}\in
r_R(c_{n-1}R)=e_{n-1}R$ and so $a_{m-1}\in e_ne_{n-1}R$. From eq.
$(2'')$, we obtain also $c_n\sigma^n(ba_{m-2})=0=c_nba_{m-2}$, so
$a_{m-2}\in e_nR$.\NL Summarizing at this point, we have $$a_m\in
e_ne_{n-1}e_{n-2}R,\quad a_{m-1}\in
e_ne_{n-1}R\quad\mathrm{and}\;a_{m-2}\in e_nR.$$\par Now in eq.
$(3)$, since
$$c_nba_{m-2}=c_{n-1}ba_{m-1}=c_nba_{m-1}=c_{n-2}ba_m=c_{n-1}ba_m=c_nba_m=0.$$
because $a_m\in r_R(c_nR)\cap r_R(c_{n-1}R)\cap r_R(c_{n-2}R)$,
$a_{m-1}\in r_R(c_nR)\cap r_R(c_{n-1}R)$ and $a_{m-2}\in r_R(c_nR)$.
By Lemma \ref{lem2}, we have
$$c_nf_{n-1}^n(ba_{m-2})=c_{n-1}f_{n-2}^{n-1}(ba_{m-1})=c_nf_{n-2}^n(ba_{m-1})
=c_{n-2}f_{n-3}^{n-2}(ba_m)$$
$$=c_{n-1}f_{n-3}^{n-1}(ba_m)=c_nf_{n-3}^n(ba_m)=0,\quad$$ because
$a_{m-2}\in r_R(c_nR)$, $a_{m-1}\in r_R(c_nR)\cap r_R(c_{n-1}R)$ and
$a_m\in\ r_R(c_nR)\cap r_R(c_{n-1}R)\cap r_R(c_{n-2}R)$. So eq.
$(3)$ becomes
$$c_n\sigma^n(ba_{m-3})+c_{n-1}\sigma^{n-1}(ba_{m-2})++c_{n-2}
\sigma^{n-2}(ba_{m-1})\eqno(3')$$ $$+c_{n-3} \sigma^{n-3}(ba_m)=0.$$
Let $b=se_ne_{n-1}e_{n-2}$ in eq. $(3')$, we obtain
$$c_n\sigma^n(se_ne_{n-1}e_{n-2}a_{m-3})+c_{n-1}
\sigma^{n-1}(se_ne_{n-1}e_{n-2}a_{m-2})$$ $$+c_{n-2}
\sigma^{n-2}(se_ne_{n-1}e_{n-2}a_{m-1})+c_{n-3}
\sigma^{n-3}(se_ne_{n-1}e_{n-2}a_m)=0.$$ By the above results, we
have
$$c_n\sigma^n(se_ne_{n-1}e_{n-2}a_{m-3})=c_{n-1}\sigma^{n-1}(se_ne_{n-1}e_{n-2}a_{m-2})$$
$$\qquad\qquad\qquad\qquad\qquad\qquad\quad=c_{n-2}
\sigma^{n-2}(se_ne_{n-1}e_{n-2}a_{m-1})=0,$$ then $c_{n-3}
\sigma^{n-3}(se_ne_{n-1}e_{n-2}a_m)=0$, so $c_{n-3}
se_ne_{n-1}e_{n-2}a_m=c_{n-3}sa_m=0$, hence $a_m\in
e_ne_{n-1}e_{n-2}e_{n-3}R$, and eq. $(3')$ simplifies to
$$c_n\sigma^n(ba_{m-3})+c_{n-1}\sigma^{n-1}(ba_{m-2})+c_{n-2}
\sigma^{n-2}(ba_{m-1})=0.\eqno(3'')$$ In eq. $(3'')$ substitute
$se_ne_{n-1}$ for $b$ to obtain
$$c_n\sigma^n(se_ne_{n-1}a_{m-3})+c_{n-1}\sigma^{n-1}(se_ne_{n-1}a_{m-2})+c_{n-2}
\sigma^{n-2}(se_ne_{n-1}a_{m-1})=0.$$ But
$$c_n\sigma^n(se_ne_{n-1}a_{m-3})=c_{n-1}\sigma^{n-1}(se_ne_{n-1}a_{m-3})=0.$$
So
$c_{n-2}\sigma^{n-2}(se_ne_{n-1}a_{m-1})=0=c_{n-2}se_ne_{n-1}a_{m-1}=c_{n-2}sa_{m-1}$.
Hence $a_{m-1}\in e_ne_{n-1}e_{n-2}R$, and eq. $(3'')$ simplifies to
$$c_n\sigma^n(ba_{m-3})+c_{n-1}\sigma^{n-1}(ba_{m-2})=0.\eqno(3''')$$
In eq. $(3''')$ substitute $se_n$ for $b$ to obtain
$$c_n\sigma^n(se_na_{m-3})+c_{n-1}\sigma^{n-1}(se_na_{m-2})=0.$$ But
$c_n\sigma^n(se_na_{m-3})=0$, so
$c_{n-1}\sigma^{n-1}(se_na_{m-2})=0=c_{n-1}se_na_{m-2}=c_{n-1}sa_{m-2}$.
Hence $a_{m-2}\in e_ne_{n-1}R$, and eq. $(3''')$ simplifies to
$$c_n\sigma^n(ba_{m-3})=0,$$ then $c_nba_{m-3}=0$. Hence $a_{m-3}\in
e_nR$.\par Summarizing at this point, we have $$a_m\in
e_ne_{n-1}e_{n-2}e_{n-3}R,\;a_{m-1}\in
e_ne_{n-1}e_{n-2}R,\;a_{m-2}\in e_ne_{n-1}R,$$ and $a_{m-3}\in
e_nR.$ Continuing this procedure yields $a_i\in eR$ for all
$i=0,1,\cdots,m$. Hence $\varphi(x)\in eR[x;\sigma,\delta]$.
Consequently, $r_S(p(x)R)\subseteq eS$.
\NL Finally, by Claims 1 and 2, we have $r_S(p(x)R)\subseteq eS\subseteq r_S(p(x)S)$. Also, since $p(x)R\subseteq p(x)S$, we have $r_S(p(x)S)\subseteq r_S(p(x)R)$. Thus $r_S(p(x)S)=eS$. Therefore $R[x;\sigma,\delta]$ is right p.q.-Baer.
\end{proof}

From Example \ref{ex3.1}, we can see that the condition ``$R$ satisfies $(\mathcal{C_\sigma})$" in
Proposition \ref{proprincipalart2} is not superfluous. On the other hand, there is an example
which satisfies all the hypothesis of Proposition \ref{proprincipalart2}.

\begin{exa}\label{ex2.3}\cite[Example 1.1]{hashemi/quasi}. Let $R_1$ be a
right p.q.-Baer, $D$ a domain and $R=T_n(R_1)\oplus D[y]$, where
$T_n(R_1)$ is the upper $n\times n$ triangular matrix ring over
$R_1$. Let $\sigma\colon D[y]\rightarrow D[y]$ be a monomorphism
which is not surjective. Then we have the following:
\NL $(i)$ $R$ is right p.q.-Baer. By \cite{birk/pqBaer}, $T_n(R_1)$
is right p.q.-Baer and hence
$T_n(R_1)\oplus D[y]$ is right p.q.-Baer.\NL$(ii)$ $R$ satisfies the condition
$(\mathcal{C_{\overline\sigma}})$. Let $\overline\sigma\colon
R\rightarrow R$ be an endomorphism defined by
$\overline\sigma(A\oplus f(y))=A\oplus\sigma(f(y))$ for each $A\in
T_n(R_1)$ and $f(y)\in D[y]$. Suppose that $(A\oplus
f(y))\overline\sigma(B\oplus g(y))=0$. Then $AB=0$ and $f(y)\sigma
(g(y))=0$. Since $D[y]$ is a domain and $\sigma$ is a monomorphism,
$f(y)=0$ or $g(y)=0$. Hence $(A\oplus f(y))(B\oplus
g(y))=0$.\NL$(iii)$ $Re$ is $\overline\sigma$-stable for all
$e\in\mathcal{S}_\ell(R)$. Idempotents of $R$ are of the form
$e_0=A\oplus 0$ and $e_1=A\oplus 1$, for some idempotent $A\in
T_n(R_1)$. Since $\overline\sigma(e_0)=e_0$ and
$\overline\sigma(e_1)=e_1$, we have the stability desired. Note that
$R$ is not reduced, and hence it is not $\overline\sigma$-rigid.
\end{exa}

\begin{cor}\label{corx}Let $(\sigma,\delta)$ be a quasi-derivation of a ring $R$. Assume that $R$ is
right p.q.-Baer, if $R$ satisfies one of the following:
\NL $(i)$ $R$ is $(\sigma,\delta)$-skew Armendariz and satisfies $(\mathcal{C_\sigma})$;
\NL$(ii)$ $\mathcal{S_{\ell}}(R)=\mathcal{B}(R)$, $\sigma(Re)\subseteq Re$ for all $e\in \mathcal{B}(R)$ and $R$
satisfies $(\mathcal{C_\sigma})$;\NL$(iii)$ $R$ is $\sigma$-rigid.\NL Then $R[x;\sigma,\delta]$ is right p.q.-Baer.
\end{cor}

\begin{proof}$(i)$ By \cite[Lemma 4]{hashemi/skew}, $Re$ is
$(\sigma,\delta)$-stable for all $e\in\mathcal{S}_\ell(R)$.
\NL$(ii)$ It follows from Proposition
\ref{proprincipalart2} and \cite[Lemma 2.3]{louzari1}.
\NL$(iii)$ Follows from Lemma \ref{lem5} and \cite[Lemma 2.5]{louzari1}.
\end{proof}

Now, we focus on the converse of Proposition \ref{proprincipalart2}.

\begin{prop}\label{propInv}Let $(\sigma,\delta)$ a quasi-derivation of a ring $R$ such that $R$ is $(\sigma,\delta)$-skew Armendariz. If $R[x;\sigma,\delta]$ is right p.q.-Baer then $R$ is right p.q.-Baer.
\end{prop}

\begin{proof}Let $S=R[x;\sigma,\delta]$ and $a\in R$. By \cite[Lemma 5]{hashemi/skew},
there exists $e\in\mathcal{S}_{\ell}(R)$, such that $r_S(aS)=eS$, in particular $aRe=0$, then $e\in r_R(aR)$, also $eR\subseteq r_R(aR)$. Conversely, if $b\in r_R(aR)$, we have $aRb=0$, then $b=ef$ for some $f=\alpha_0 +\alpha_1x+\alpha_2x^2+\cdots+\alpha_nx^n\in S$, but $b\in R$, thus $b=e\alpha_0$. Therefore $b\in eR$. So that $r_R(aR)=eR$.

\end{proof}

\begin{thm}\label{main/th}Let $(\sigma,\delta)$ a quasi-derivation of a ring $R$ such that $R$ is $(\sigma,\delta)$-skew Armendariz and satisfies $(\mathcal{C_\sigma})$. Then $R$
is right p.q.-Baer if and only if $R[x;\sigma,\delta]$ is right p.q.-Baer.
\end{thm}

\begin{proof}It follows immediately from Corollary \ref{corx} and
Proposition \ref{propInv}.
\end{proof}

\begin{cor}[{\cite[Corollary 15]{hong/2000}}]\label{cor/rigid}Let $R$ be a ring, $\sigma$ an
endomorphism and $\delta$ be a $\sigma$-derivation of $R$. If $R$ is
$\sigma$-rigid, then $R$ is right p.q.-Baer if and only if
$R[x;\sigma,\delta]$ is right p.q.-Baer.
\end{cor}

From Example \ref{ex2.2}, we see that Theorem \ref{main/th} is a
generalization of \cite[Corollary 15]{hong/2000}. There is an
example of a ring $R$ and a quasi-derivation $(\sigma,\delta)$,
which satisfies all the hypothesis of Theorem \ref{main/th}.

\begin{exa}[{\cite[Example 3.11]{louzari1}}]Let $R=\C$ where $\C$ is the field of complex numbers.
Define $\sigma\colon R\rightarrow R$ and $\delta\colon R\rightarrow
R$ by $\sigma(z)=\overline{z}$ and $\delta(z)=z-\overline{z}$, where
$\overline{z}$ is the conjugate of $z$. $\sigma$ is an automorphism
of $R$ and $\delta$ is a $\sigma$-derivation. We have \NL$(i)$ $R$
is Baer (so right p.q.-Baer) reduced; \NL$(ii)$ $R$ is
$\sigma$-rigid, then it is $(\sigma,\delta)$-skew Armendariz and
satisfies $(\mathcal{C_\sigma})$.
\end{exa}

\begin{rk} Example \ref{ex2.1}, shows that Theorem \ref{main/th} is not a consequence of \cite[Corollary 2.8]{hashemi/quasi}.
\end{rk}
\section*{acknowledgments}
This work was supported by the integrated action Moroccan-Spanish A/5037/06. The second author wishes to thank Professor Amin Kaidi of Universidad de Almer\'ia for his generous hospitality.

\end{document}